\documentclass[openright,a4paper,american,11pt]{amsart}
\usepackage[margin = 3cm]{geometry}
\usepackage[latin1]{inputenc}
\usepackage{graphicx}
\usepackage{psfrag}
\usepackage{todonotes}
\usepackage{amsmath}
\usepackage{amssymb}
\usepackage{babel}
\usepackage{amsfonts}
\usepackage{url}

\usepackage{setspace}

\input xy

\xyoption{all}

\usepackage{boxedminipage}

\usepackage{color}

\newcommand{\ZZ}{{\mathbb Z}}

\newcommand{\RR}{{\mathbb R}}

\renewcommand{\P}{\mathcal P}

\newcommand{\ext}{\mbox{Ext\,}}
\renewcommand{\int}{\mbox{Int\,}}
\newcommand{\Conv}{\mbox{Hull\,}}
\renewcommand{\S}{\mbox{s\,}}
\renewcommand{\epsilon}{\varepsilon}

\newtheorem{thm}{Theorem}[section]

\newtheorem{prop}[thm]{Proposition}
\newtheorem{lemma}[thm]{Lemma}

         {\theoremstyle{definition}

\newtheorem{rem}[thm]{Remark}}

         {\theoremstyle{definition}

\newtheorem{exa}[thm]{Example}}

\begin{document}
\onehalfspacing 

\title[Maximally inflected ovals]{
  On maximally inflected hyperbolic curves}

\date{\today}

\author[A. Arroyo]{Aubin Arroyo}
\author[E.  Brugallé]{Erwan Brugallé}

\address{Erwan Brugallé, 
École polytechnique,
Centre Mathématiques Laurent Schwatrz, 91 128 Palaiseau Cedex, France }

\email{erwan.brugalle@math.cnrs.fr}

\author[L. López de Medrano]{Lucia López de Medrano}

\address{Unidad Cuernavaca del Instituto de Matemáticas,Universidad Nacional
Autonoma de México. Cuernavaca, México}

\email{aubinarroyo@im.unam.mx,  lucia.ldm@im.unam.mx }



\begin{abstract}
In this note we study the distribution of real inflection points among
the ovals of a real non-singular hyperbolic curve of even degree.  
Using Hilbert's method
we show that for any 
integers $d$ and $r$ such that $4\leq r \leq 2d^2-2d$, there is a
non-singular hyperbolic curve of degree $2d$ in $\RR^2$
with exactly $r$ line segments in the boundary of its convex hull. 
We also give a complete
classification of possible distributions of inflection points among the ovals of a maximally inflected non-singular hyperbolic curve of degree $6$. 
\end{abstract}

\maketitle

\newcommand{\RP}{\mathbb{RP}^2}
\newcommand{\CP}{\mathbb{CP}^2}

\section{Introduction}
The fact that a non-singular real algebraic
curve in $\RP$ of degree $\delta$ has at most 
$\delta(\delta-2)$ real inflection points
was proved by Klein in 1876, see \cite{Kle76}; see also \cite{Ron98}, \cite{Sch04} and \cite{Vir88}. 
A non-singular real plane algebraic curve is  called 
 \emph{maximally inflected} if it has $\delta(\delta-2)$ distinct real inflection points. 
The existence of maximally inflected curves of any degree was also
proved by Klein.  
Possible distributions of inflection points 
of a maximally inflected 
real algebraic curve in $\RP$ are
subject to non-trivial obstructions that mainly remain mysterious (see
for example \cite{KS03} and \cite{BLdM12}).
In this note we focus on the case
of maximally inflected hyperbolic curves of even degree.

Let $X=\RR^2$ or $\RP$.
A real algebraic curve $C$ in $X$ is said to be
\emph{hyperbolic}
 if there exists a point
$p\in X\smallsetminus C$ such that any real line through $p$
intersects $C$ only in real points. 
The topology of a non-singular hyperbolic curve is easy to describe.
An embedded circle $O$ in $X$ is called an \emph{oval} if its
complement in $X$ has two connected components. In this
case, one of them, called the \emph{interior} of $O$ and denoted
by $\int(O)$, is
homeomorphic to a disk, while the other is  called the \emph{exterior} of
$O$ and denoted
by $\ext(O)$.
An oval $O$ is said to be contained in another oval $\tilde O$ if $O\subset \int(\tilde O)$.
A non-singular hyperbolic  curve $C$ of  degree $2d$
is a set of $d$ nested ovals, i.e. the
inclusion relation among its ovals is a total ordering. 
The oval of $C$ containing all the others  is called the \emph{outer
  oval}
 of $C$. 
We define the \emph{interior} and the \emph{exterior} of $C$, denoted
by $\int(C)$ and $\ext(C)$, as the interior and exterior of its outer
oval respectively.

\subsection{Line segments in the boundary of the convex hull of a
  hyperbolic curve in $\RR^2$}\label{sec:intro line}
Denote by $\Conv(S)$ the convex hull of a compact set $S\subset \RR^2$. Of course, if $S$ is not convex then the boundary of $\Conv(S)$ contain some line segments. Denote by $s(S)$ the number of line segments in $\partial \Conv(S)$.

In the case of a  non-singular hyperbolic  curve $C$ of  degree $2d$
in $\RR^2$, the number $\S(C)$ is related to the number of inflection points contained in the outer oval $O$ of $C$. 
More precisely, for each line segment $l$ in $\partial \Conv(C)$, there
are  at least two inflection points of $C$
in the connected component of the closure $O \setminus \partial \Conv(C)  $ with the same endpoints
than $l$. Klein Inequality then implies that:
\begin{equation*}
\S(C) \leq 2d(d-1).
\end{equation*} 

Next theorem shows, in particular, that this upper bound is sharp,
answering to a question posed by De Loera, Sturmfels, and Vinzant in
\cite{DLSV}. 
\begin{thm}\label{T2}
Let $d$ and $r$ be two positive integers such that $4\leq r \leq 2d(d-1)$. Then there exists a 
non-singular  maximally inflected hyperbolic curve $C$  of degree $2d$ in $\RR^2$ with  $\S(C) = r$. 
\end{thm}

The proof of Theorem \ref{T2} is 
given in Section \ref{sec:hilbert}, and is based on Hilbert's method
of construction of real algebraic curves. See \cite{V4} for an exposition
of this method in modern terms.
For the reader's convenience, we break the proof of Theorem \ref{T2} into two parts:
in Section \ref{sec:hilbert2} we explain how to adapt Hilbert's method to construct
non-singular maximally inflected hyperbolic curves of even degree with
all real inflection points contained in the outer oval;
In Section \ref{sec:hilbert3} we adapt this construction to prove Theorem \ref{T2}.

\subsection{Hyperbolic curves of degree 6}
Recall that a hyperbolic curve of degree 6 has  three nested ovals. 
Bezout's Theorem implies that there are no inflection points in the smallest oval of $C$.  
Next theorem gives then a complete classification of possible
distributions of inflection points among the ovals of a non-singular
maximally inflected hyperbolic curve of degree 6 in $\RP$.
  
\begin{thm}\label{theo:hyp deg 6}
Let $C$ be a non-singular maximally inflected hyperbolic curve of
degree 6 in $\RP$. Then, the outer oval of $C$ contains at least 6 real inflection points. 
Moreover, for any 
$0 \le k \le 9$, there exists such a hyperbolic curve with exactly $6+ 2k$ real inflection points on the outer oval.
\end{thm}

The proof of this theorem will be given in Section \ref{sec:degree 6} and it combines two
main ingredients: 
First  we use the tropical methods developed in
\cite{BLdM12} to construct real algebraic curves with a prescribed
position of their real inflection points.
Then we rely on Orevkov's braid theoretical method to prove that the outer oval of $C$ cannot contain less than 6 inflection points. 
Note that our proof of this latter fact uses in a crucial way that we deal with curves
of degree 6. 
The existence of a non-singular maximally inflected hyperbolic curve of degree $2d>6$ with an outer oval not containing any inflection points
remains an open problem.

\medskip
\noindent \textbf{Acknowledgments:} We thank the Laboratorio
Internacional Solomon Lefschetz (LAISLA), associated to the CNRS
(France) and CONACYT (Mexico). The first author was partially
supported by CONACyT-México \#167594. The third author was partially
supported by UNAM-PAPIIT IN-117110. We also thank Bernd Sturmfels for
his 
stimulating questions, and, as usual, Jean-Jacques Risler for valuable
comments on a preliminary version of this paper. 

\section{Construction of inflected curves}\label{sec:hilbert}
Hilbert's method can be adapted to construct non-singular maximally inflected hyperbolic curves of even degree with all real inflection points contained in the outer oval.
The proof of Theorem \ref{T2}, given in Section \ref{sec:hilbert3}, will consider perturbations of certain curves with generic nodes; these curves are described in Section \ref{sec:hilbert2}. 
Recall that a \emph{node} of an algebraic curve $C$ in $\CP$ is a
point 
in the curve $C$ which is the transverse intersection of two
non-singular local
branches of $C$. 
Such a node is called \emph{generic} in the case that both branches of $C$ at the point have intersection multiplicity 2 with their tangent at the node.
If $C$ is a real curve with a node $p \in \RP$, then either $p$ is the intersection of two real branches or is  the intersection of two  complex conjugated branches of $C$. 
For our construction we shall focus only on nodes of the former type.

The fact that Hilbert's method produces maximally inflected curves is
based on the following observation (see Figure \ref{fig:smoothing node}).
\begin{prop}[see {\cite[Proposition 3.1]{Ron98}}]\label{prop:klein}
Let $C$ be a real algebraic curve in $\RP$ having a generic node
$p\in\RP$
 with two real branches. Then any
 real perturbation of $C$ has exactly two real inflection points
on a neighborhood of $p$. 
\end{prop}

\begin{figure}[htbp] 
   \centering
\begin{tabular}{ccccc}
   \includegraphics[width=3cm]{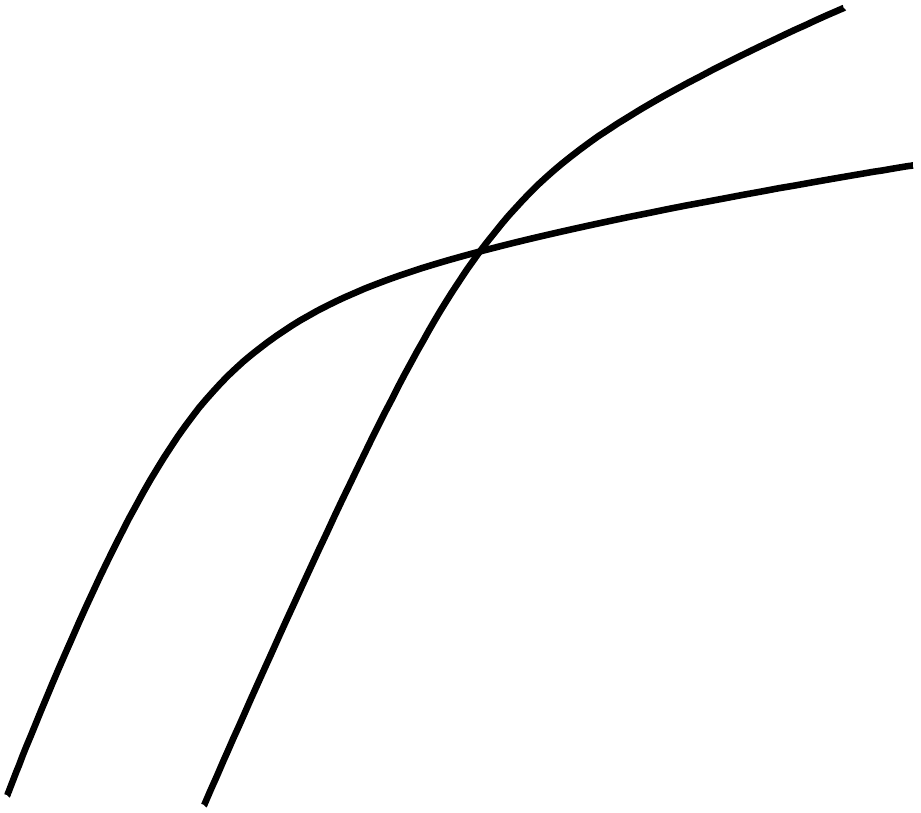} & \hspace{1cm} &
   \includegraphics[width=3cm]{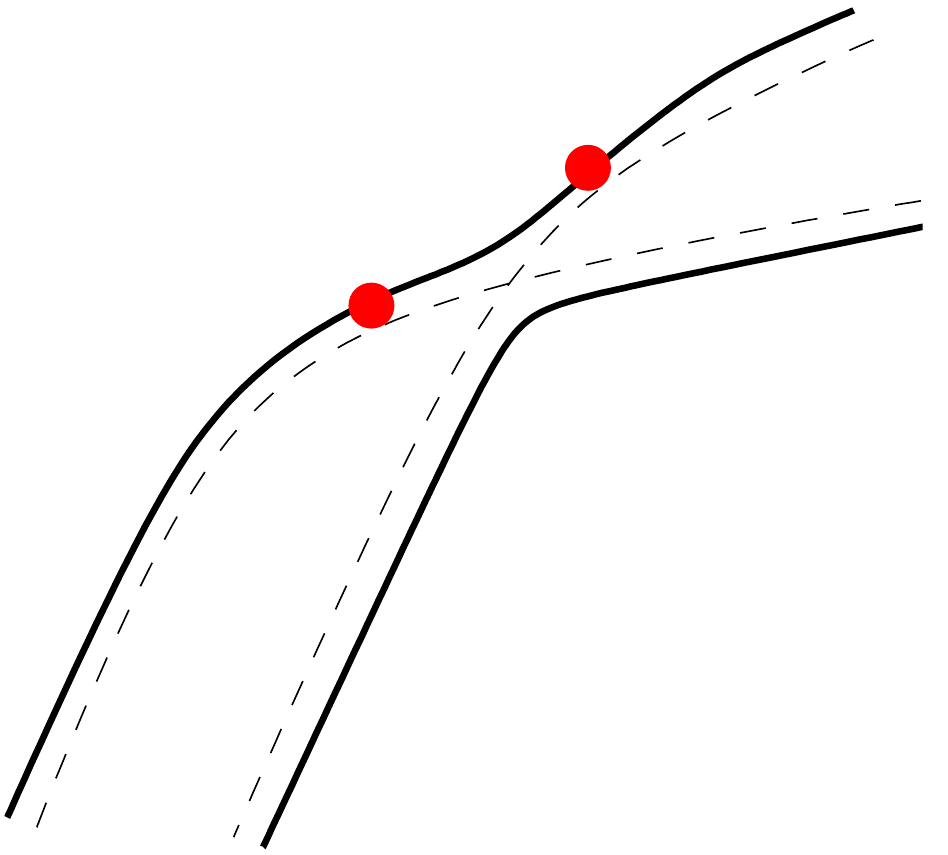} & \hspace{1cm} &
   \includegraphics[width=3cm]{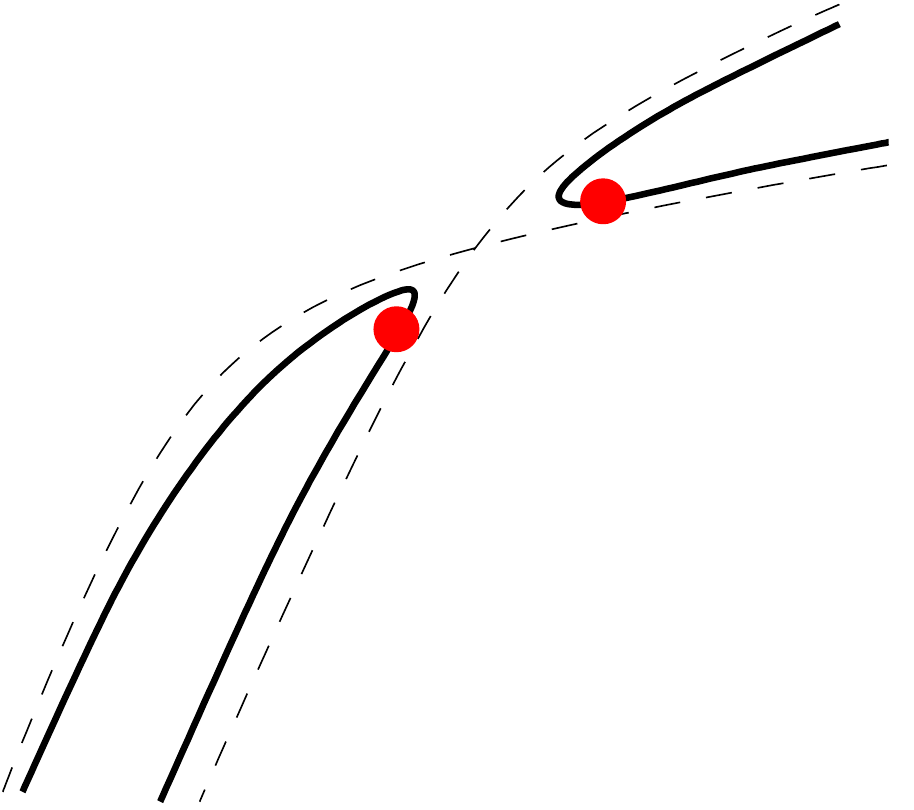}
\\ \\ a) && b) &&c)
\end{tabular}
   \caption{Smoothing of a generic node and the creation of two real inflection points.}
   \label{fig:smoothing node}
\end{figure}

A real  algebraic curve $C$
in $\RP$ is said to be \emph{generically inflected} if it is
non-singular or nodal, and
any of its
real inflection point is 
a non-singular point at which
$C$ has intersection multiplicity 3 with its tangent. Given a generically inflected real algebraic curve $C$, we
 denote by
$I(C)$  the set of its non-singular real  inflection points, 
and by $N(C)$ the set of 
its real generic nodes.
It follows from Klein Inequality and Proposition \ref{prop:klein} that if $C$ has degree $d$,
\begin{equation}\label{equation} \#I(C)+2\#N(C)\leq d(d-2).\tag{*}
\end{equation}

A generically inflected curve $C$ is called \emph{maximally inflected} if
equality holds in \eqref{equation}.

Next Lemma is an immediate consequence of 
Proposition \ref{prop:klein}. 

\begin{lemma}\label{Kharlamov}
Let $C$ be a generically inflected
 real algebraic curve in $\RP$.
 Then for any 
non-singular real 
small deformation $C'$ of $C$, one has
\begin{equation*}
\#I(C')= \#I(C)+2\,\# N(C).
\end{equation*} 
In particular, any 
 non-singular 
real 
small deformation of a maximally inflected real curve is  maximally
inflected.
\end{lemma}

\subsection{Construction of hyperbolic curves with a maximally inflected outer oval}\label{sec:hilbert2}
For the rest of this section we fix a non-singular real ellipse $C_0$
in $\RR^2$. 
An oval $O$ of a maximally inflected non-singular curve $C$ is said to be maximally inflected, if $I(C)\subset O$.  
Let us consider the following properties for a non-singular maximally inflected hyperbolic curve $C$ of degree $2d$:
\begin{itemize}
\item $(\P)$:  $C$ has a maximally inflected outer oval and $\int(C)\cap\int(C_0)$ is convex; 
\item $(\P')$: $C\cap C_0$ is a set of $4d$ distinct real points 
lying on the outer oval of $C$,  none of them being an inflection point of $C$.
\end{itemize}

Next we proceed by induction to construct 
a family $(C_d)_{d\ge 1}$ of real algebraic curves 
with the above properties of degree $2d$ in $\RR^2$. Given two real algebraic curves $C$ and $L$ in $\RR^2$,
respectively, defined by  equations $P(x,y)=0$ and $Q(x,y)=0$,  a 
\emph{perturbation of $C$ by $L$} is a real algebraic curve
defined by the equation $P(x,y)+ \epsilon Q(x,y)=0$, with 
$\mid\epsilon\mid\ll 1$ a
real number. 
Note that two  perturbations of $C$ by $L$ using 
 parameters with opposite signs have a priori different topology
in $\RR^2$.
The family $(C_d)_{d\ge 1}$ is constructed as follows:

\begin{enumerate}
\item We choose for $C_1$ any non-singular  ellipse in $\RR^2$
intersecting $C_0$ in 4 real points.

\item Suppose that the curve  $C_{d-1}$ is constructed, and does not
  contain $C_0$ as a component. Consider
  $L_{d}$ a union of $2d$ real lines in $\RR^2$  
 intersecting $C_0$ in a set $P_{d}$ of $4d$ distinct real points such
 that $P_{d} \subset \ext(C_{d-1})$, and that any connected component
 of $C_0 \cap \ext(C_{d-1})$ contains an even number of points of
 $P_{d}$. These two conditions ensure that there exists a 
 perturbation of $C_{d-1}\cup C_0$
 by $L_d$ producing a non-singular hyperbolic curve, that we define to
 be $C_d$.
\end{enumerate}

\begin{prop}\label{prop:induction}
If $C_{d-1}$ satisfies properties $(\P)$ and $(\P')$, then 
$C_{d}$ also satisfies properties $(\P)$ and $(\P')$.
\end{prop}
\begin{proof}
By assumption we have:
 $$\#I(C_{d-1})+2\,\#N(C_0\cup C_{d-1})=2(d-1)(2(d-1)-2)+8(d-1)=2d(2d-2).$$
Hence $C_0 \cup C_{d-1}$ is a maximally inflected real hyperbolic 
curve. Property $(\P')$ implies that this latter curve is generically
inflected. Lemma \ref{Kharlamov} implies that $C_d$ is maximally
inflected.
Moreover, property $(\P)$ implies that 
$I(C_{d-1})\subset \ext(C_0)$ and that each node of $C_0 \cup C_{d-1}$
produces two real inflection points of $C_d$ on the outer oval of
$C_d$ (see Figure \ref{fig:smoothing node}b).
Hence $C_d$ has a maximally inflected outer oval, and
$I(C_{d})\subset \ext(C_0)$. 

To prove that  $\int(C_d)\cap\int(C_0)$
is convex, it is enough to show that $\int(C_d)$ is locally convex at
each point of the outer oval of $C_d$ contained in $\int(C_0)$. 
However this is true for points  close to $P_d$, since $C_d$
is a perturbation of $C_0$ in a neighborhood of a point in $P_d$. 
Moreover, this is true for all points since
 $I(C_{d-1})\subset \ext(C_0)$; and hence 
$\int(C_d)\cap\int(C_0)$ is convex.

By construction we have $C_d\cap C_0=P_d$ which has
cardinality $4d$. 
Since $I(C_{d-1})\cap C_0=\emptyset$, 
 there exists a small neighbourhood $U$ of $I(C_{d-1})$
 such that $U\cap C_0=\emptyset$.  
Let $U'$ be a neighbourhood of 
$P_{d-1}$ such that 
$P_d \cap U' = \emptyset$. 
Then 
 $I(C_d)\subset U\cup U'$, so $C_d$ satisfies property $(\P')$
\end{proof}

Note that $C_1$ obviously satisfies properties $(\P)$ and
$(\P')$. Hence Proposition \ref{prop:induction} ensures the existence
of  a family $(C_d)_{d\geq 1}$ containing only  non-singular  hyperbolic 
curves in $\RR^2$ with a maximally inflected outer oval.

\subsection{Proof of Theorem \ref{T2}}\label{sec:hilbert3}
Here we explain how to control the numbers $\S(C_d)$ in the
previous construction (see  section \ref{sec:intro line} for the definition of $\S(C)$).
For this, at each step we require the additional condition 
that the set $P_d$ is disjoint from the endpoints of line segments in
$\Conv(C_{d-1}\cup C_0)$ (see item
(2) in section \ref{sec:hilbert2} for the definition of $P_d$). We denote by $r_{d}$ the number of points in
$P_d$ contained in $\partial \Conv(C_{d-1}\cup C_0)$.
Note that we may freely choose the number $r_d$, $d\ge 2$, between $0$
and $4d$.

\begin{prop}
For any integer $d\geq 1$, one has
 $\S(C_{d} )=\S(C_{d-1}\cup C_0)$ and  
 $\S(C_{d}\cup C_0 )=\S(C_d)+r_d$.

\end{prop}
\begin{proof}
By 
stability, we have
$\S(C_d) \geq \S(C_{d-1}\cup C_0)$. If this inequality were
strict, there would exists a line segment in $\partial\Conv(C_d)$ coming from the
perturbation of a real point of $C_{d-1}$ at which $C_{d-1}$ has intersection
multiplicity at least 4 with its tangent. 
However the curve $C_{d-1}$ is maximally inflected, and in particular
 is generically inflected. Hence all this together 
 proves that 
$\S(C_d) = \S(C_{d-1}\cup C_0)$.

There are two types of line segments contained in $\partial \Conv(C_{d}\cup
C_0)$: those which are a perturbation of a line segment contained in
$\partial \Conv(C_{d-1}\cup C_0)$, and those which come from the perturbation
of $C_{d-1}\cup C_0$ in a neighborhood $U_p$ of a point
$p\in P_d \cap \partial \Conv(C_{d-1}\cup C_0) $; see Figure \ref{fig:otrodibujo}.
The set $\partial\Conv\left((C_d\cup C_0)\cap U_p\right)$
contains a line segment. Hence the set $\partial\Conv\left(C_d\cup
C_0\right)$ contains exactly $r_d$  line segments more than
$\partial\Conv\left(C_{d-1}\cup
C_0\right)$.
 \begin{figure}[!] 
   \centering
   \includegraphics[width=5in]{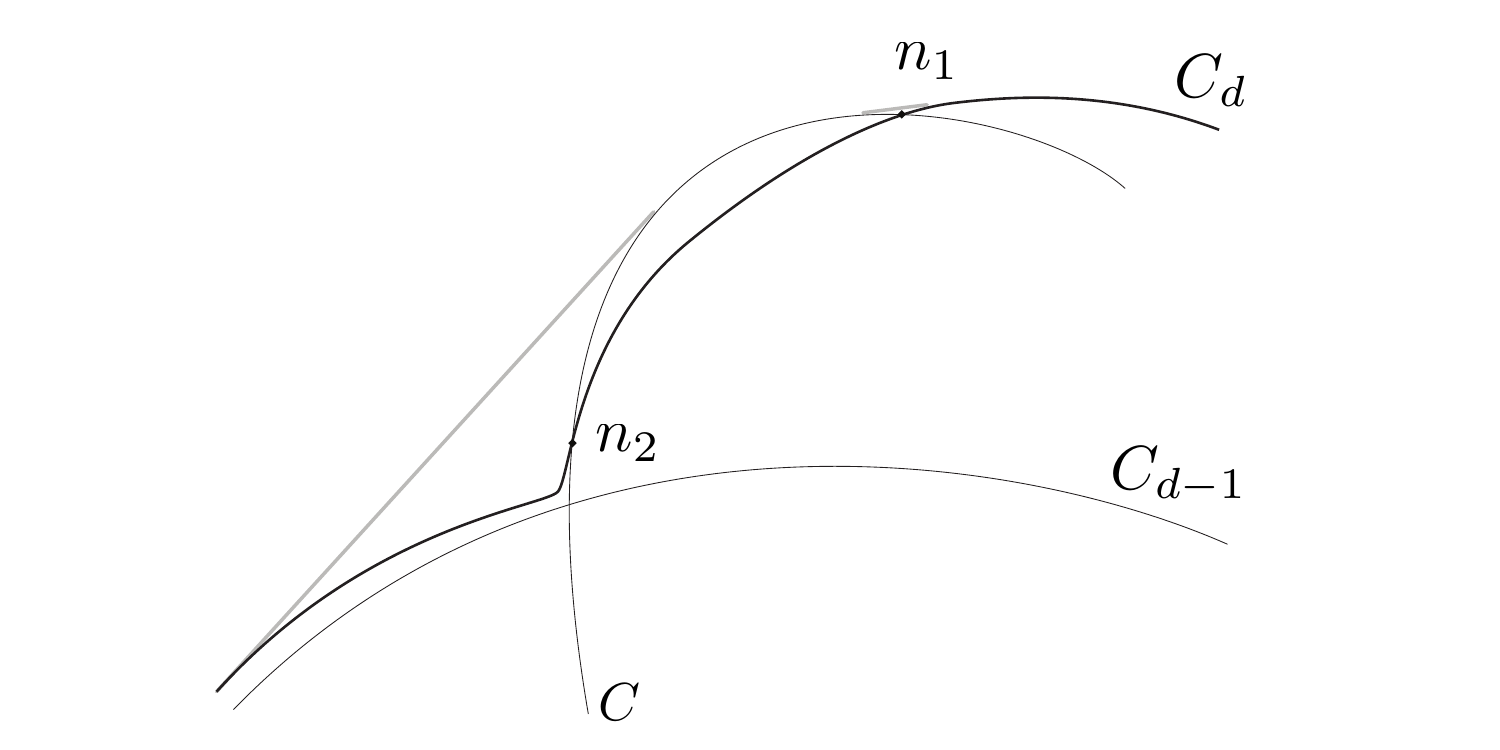} 
   \caption{Gray lines correspond to  line segments in  $\partial\Conv(C_0\cup C_d)$. The node $n_1$ is a point in $P_d$ containted in $\partial \Conv(C_{d-1}\cup C_0)$. The node $n_2$ is a point of $P_d$ contained in the interior of $ \Conv(C_{d-1}\cup C_0)$.}
   \label{fig:otrodibujo}
\end{figure}
\end{proof}

Now, Theorem \ref{T2} follows from the fact that we can choose the number $r_d$ between $0$ and $4d$ freely, and because  $\S(C_{1})=0$ and  $\S(C_{1}\cup C_0 )=4$.

\begin{exa}
Let us 
give
an explicit 
equation of
 a degree 6 polynomial  
defining a non-singular hyperbolic curve of degree $6$ in $\RR^2$
with 12 line segments on the boundary of its convex hull.
For that, consider three numbers $\epsilon_{1} =1.58$, $\epsilon_{2}=
-5.5\times 10^{-3}$ and $\epsilon_{3} = -10^{-10}$, and denote by
$C:=C(x,y)= x^2+y^2-1$. 
Consider the following polynomial:
$$
f(x,y) := {C}^{3}+{C}^{2}\left( {y}^{2}-{a}^{2}{x}^{2} \right) \epsilon_{{1}}+
C \left( {y}^{2}-{x}^{2} \right)  \left( {y}^{2}-{b}^{2}{x}^{2} \right) \epsilon_{{2}}
+  \epsilon_{{3}},
$$
where $a=\tan(\pi/12)$ and $b=\tan(5\pi/12)$. Figure \ref{fig:lacurvadegradoseis} is a numerical plot of the zero locus of $f$.
In the construction of this example we used two lines through the origin at angles $\pm \frac{\pi}{12}$  in the first step, and four lines through the origin at angles $ \pm \frac{\pi}{4}, \pm \frac{5\pi}{12} $ in the second step (all angles are measured with respect to the positive $x$-axis).
\begin{figure}[htbp] 
   \centering
   \includegraphics[width=3in, angle = -90]{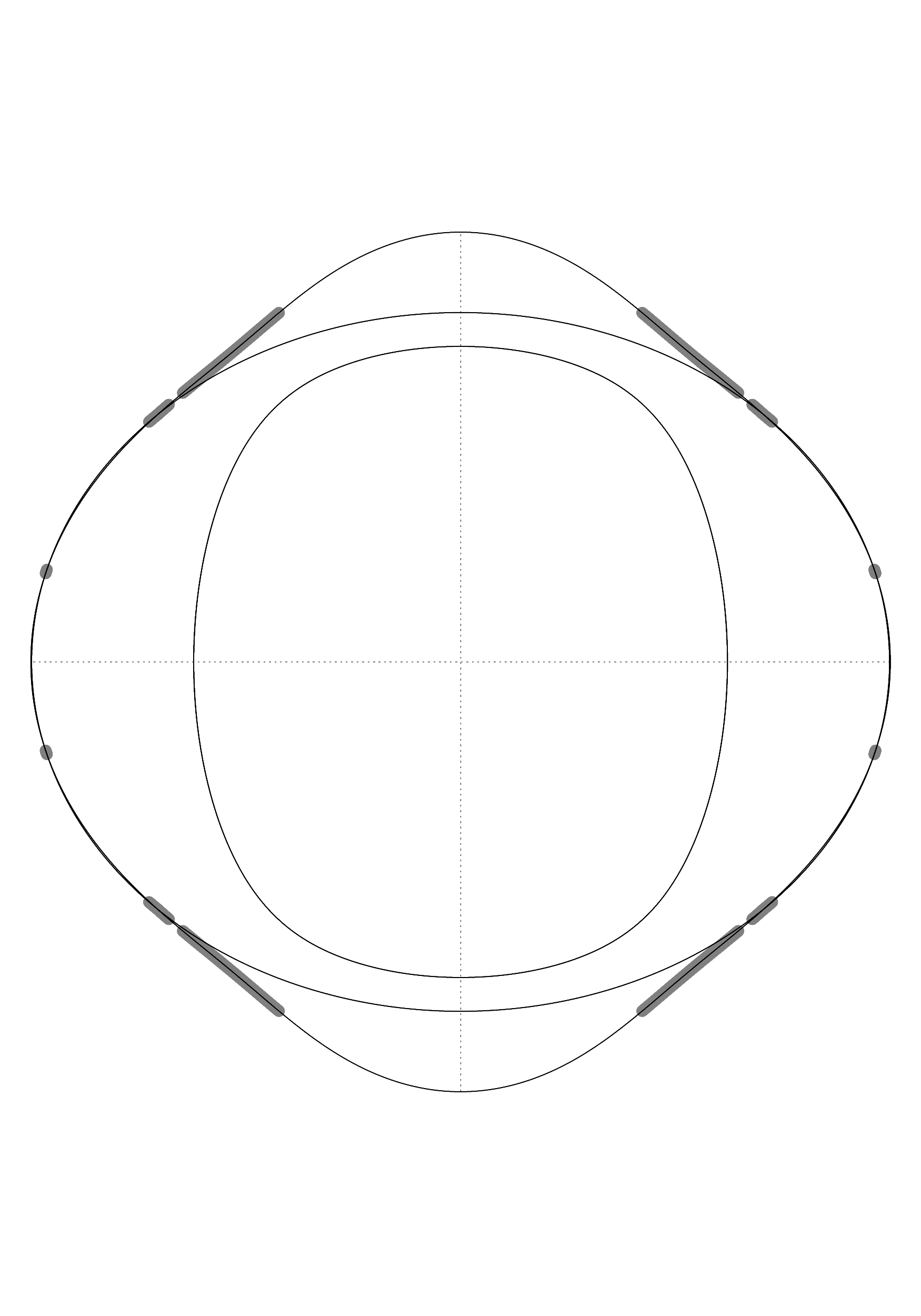} 
   \caption{A  non-singular hyperbolic curve of degree 6 in $\RR²$ with exactly 12 line segments in the boundary of its convex hull. Shaded regions on the outer oval (three in each quadrant) correspond to points in the curve contained in the interior of the convex hull of the curve. }
   \label{fig:lacurvadegradoseis}
\end{figure}
The plot in Figure \ref{fig:lacurvadegradoseis} is obtained by finding numerically the roots of the real one variable polynomial defined by the intersection of the curve with lines through the origin.
\end{exa}

\section{Maximally inflected hyperbolic curves of degree
  6}\label{sec:degree 6}

Now we turn to the study all possible arrangements of the $24$ real inflection
points of a 
non-singular maximally inflected hyperbolic sextic $C$ in 
$\RP$.
We denote by $O_1,O_2,O_3$ the three ovals of $C$ in such a way
that $O_1$ contains $O_2$ which in its turn contains $O_3$.
By
Bezout Theorem, the oval $O_3$  does
not contain any real inflection point. Moreover, each oval $O_1$ and
$O_2$ contains an even number of real inflection points.

\subsection{Obstructions}
Here we prove that the oval $O_2$ cannot contain
more than $18$ real 
inflection points.  Proposition  \ref{prop:min 6} is a straightforward
consequence of
Lemmas \ref{lem:prohib} and \ref{lem:bitangent}, that we prove in the
rest of this section.

\begin{prop}\label{prop:min 6}
Let $C$ be a non-singular maximally inflected algebraic hyperbolic curve of
degree 6 in $\RP$. Then the oval $O_1$ contains at least 6 real inflection points.
\end{prop}

Let $C_0$ and $C_1$ be two non-singular real algebraic curves in 
$\RP$ of degree at least 3, such that there exists a continuous path $(C_t)_{0\le t\le 1}$ of 
non-singular real algebraic curves in 
$\RP$ from $C_0$ to $C_1$.
A real line $D$ is said to be \emph{deeply tangent} to the curve $C_t$
if there exists a point $p\in D\cap C_t$ such that the order of
contact of $C_t$ and $D$ is at least $4$. 

Recall that if $C_0$ and $C_1$ are
generic, then they do not have any deep tangent lines. Moreover Klein
proved in \cite{Kle76} (see also \cite{Ron98}) that in this
case, it is always possible to choose the family $(C_t)$ such that 
appearance and disappearance of real inflections points when $t$
varies
 occur only when passing through a curve $C_{t_0}$ with a unique deep
 tangent $D$ with order of contact exactly $4$ with $C_{t_0}$ at a unique
 point $p$ (see Figure \ref{fig:deformation}b). 
Then for $t=t_0 +\varepsilon$ with
 $-1\ll \varepsilon \ll 1$ a  non-zero real number, the line $D$ deforms to a
 real bitangent of $C_t$ tangent to $C_t$ in two real points close to $p$ (see Figure \ref{fig:deformation}a),
 and deforms to a
 real bitangent of $C_t$ tangent to $C_t$ to two complex conjugated
 points for $t=t_0 -\varepsilon$ (see Figure \ref{fig:deformation}c). 
The curve $C_{t_0 +\varepsilon}$ has two real inflection points close to $p$, that
 disappear when passing through $C_{t_0}$.
\begin{figure}[htbp] 
   \centering
\begin{tabular}{ccccc}
   \includegraphics[width=3.5cm]{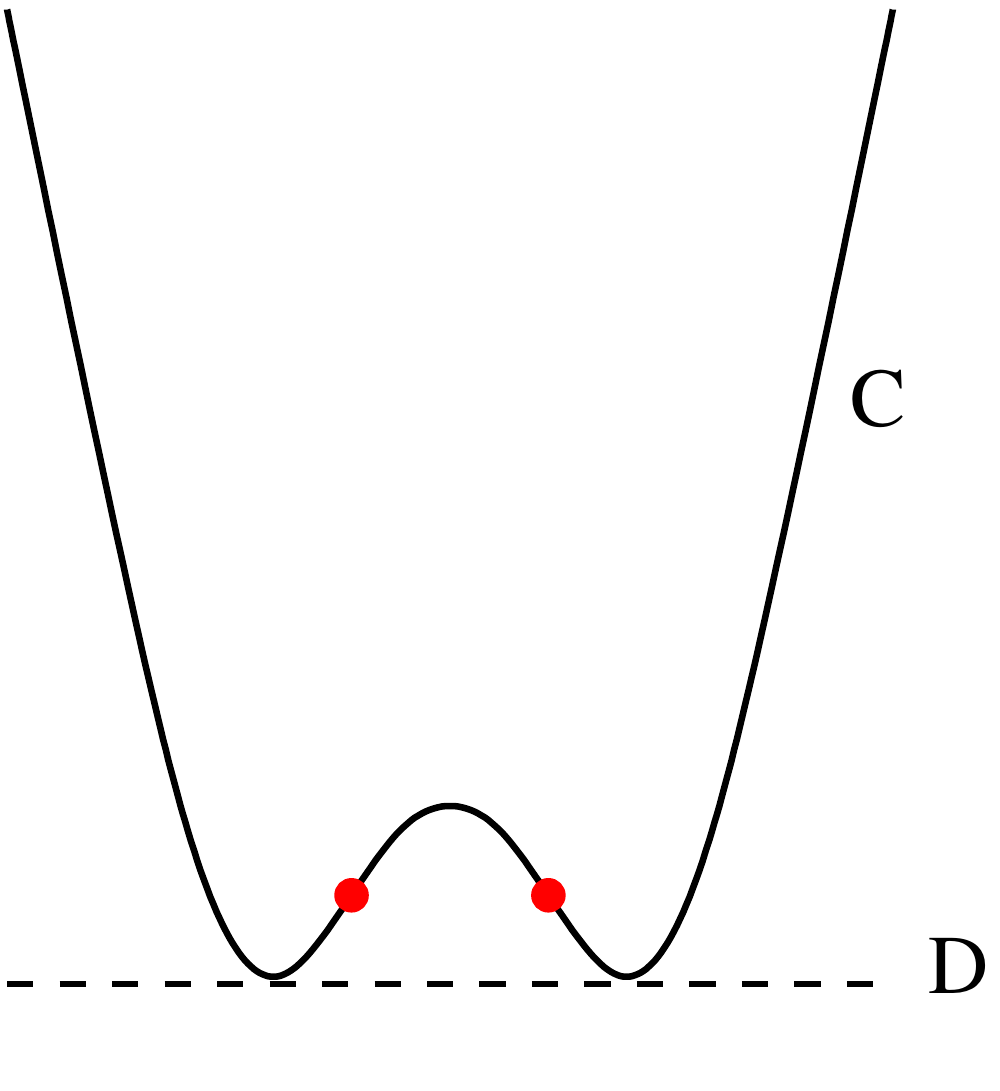} & \hspace{1cm} &
   \includegraphics[width=3.5cm]{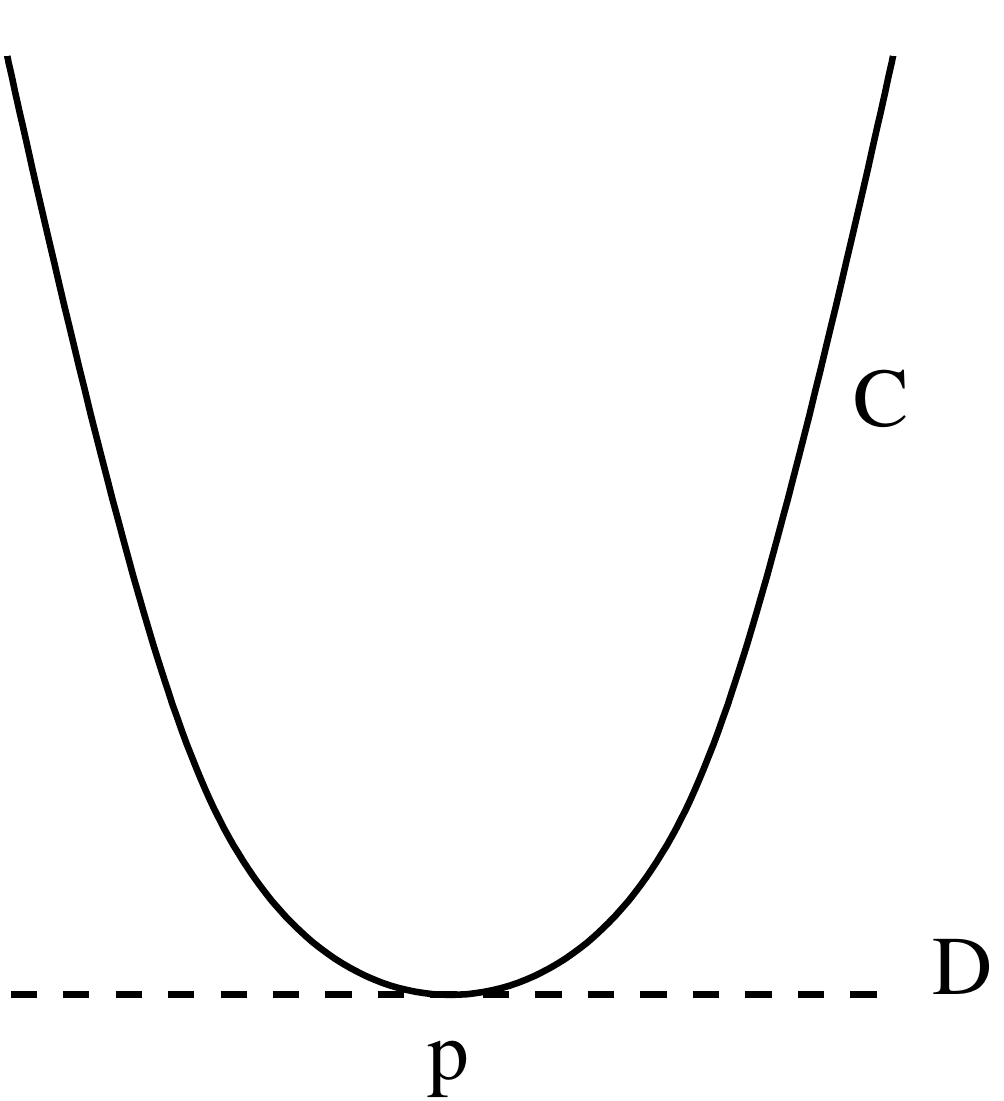} & \hspace{1cm} &
   \includegraphics[width=3.5cm]{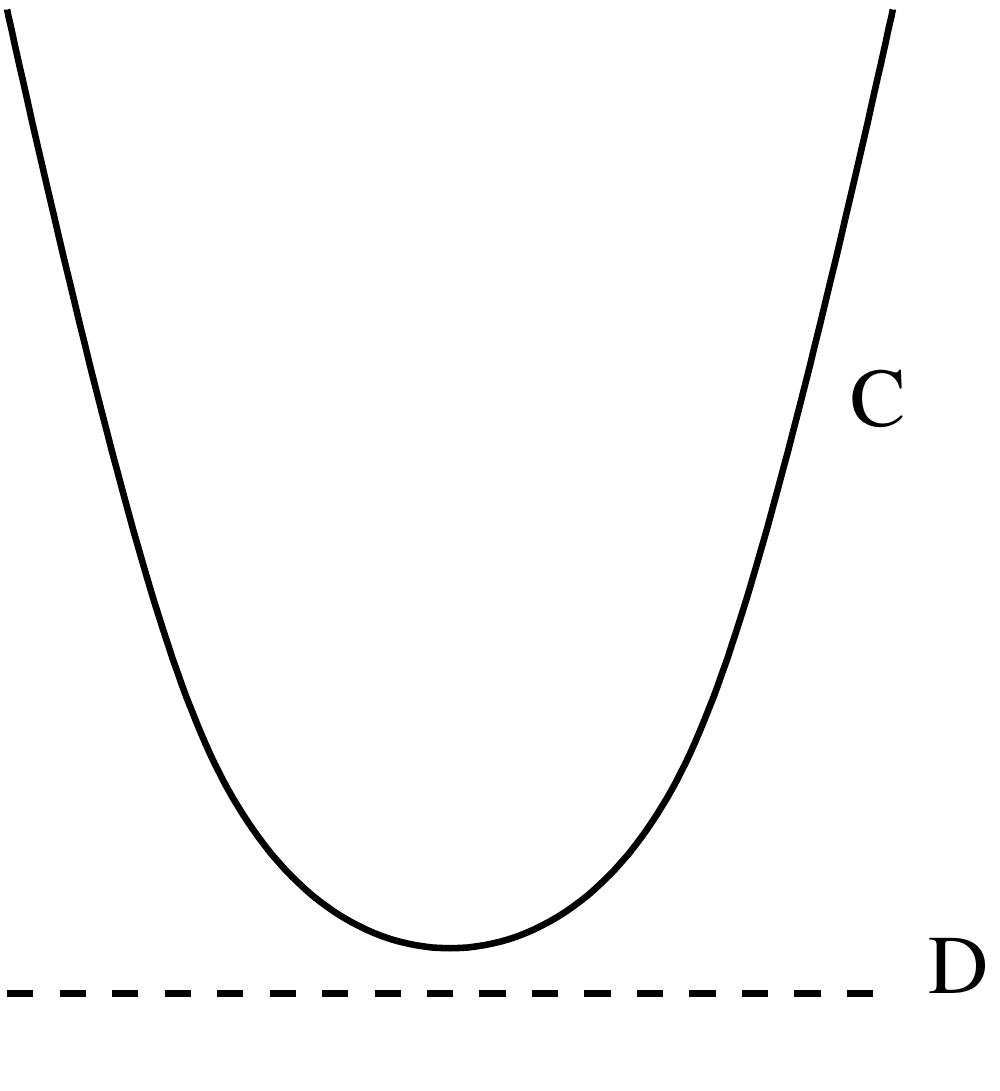}
\\ \\ a) $C=C_{t_0+\varepsilon}$ && b) $C = C_{t_0}$ && c) $C=C_{t_0-\varepsilon}$ 

\end{tabular}
   \caption{Appearance/disappearance of real inflection points}
   \label{fig:deformation}
\end{figure}

\begin{lemma}\label{lem:bitangent}
Let $C$ be a non-singular maximally inflected hyperbolic curve of
degree 6 in $\RP$ such that the oval $O_i$ contains exactly $2n$ real inflection
points. Then the curve $C$ has exactly $n$ real bitangents which are
tangent to $C$  at two
 points on $O_i$.
\end{lemma}
\begin{proof}
Let $C_0$ be the non-singular hyperbolic real sextic in $\CP$ 
obtained by a generic perturbation of the union of $3$
nested disjoint ellipses. By
construction, $C_0$ does not have any real inflection points, and does
not have any real bitangent tangent to two real points. 
It follows from Bezout
Theorem that the number of real bitangents 
is invariant under deformation in the space of non-singular
real algebraic curves of degree 6. 
Since the space of non-singular hyperbolic  sextics in $\RP$ is
connected (see \cite{Nui68}), the result follows from
 the description of appearance/disappearance of real
inflection points given above.
\end{proof}

\begin{figure}[htbp] 
   \centering
\begin{tabular}{ccc}
   \includegraphics[width=6cm]{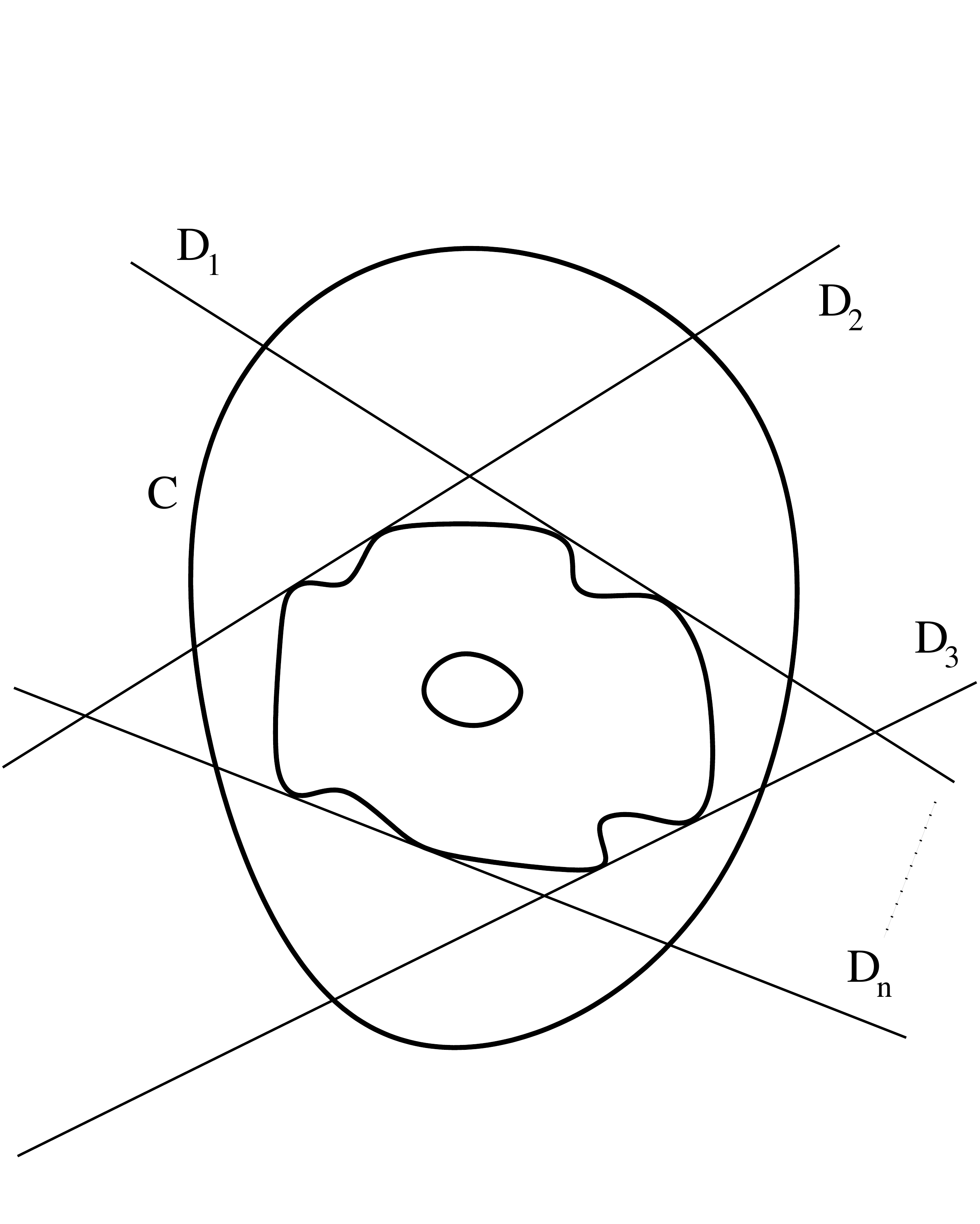} & \hspace{3ex} &
   \includegraphics[width=6cm]{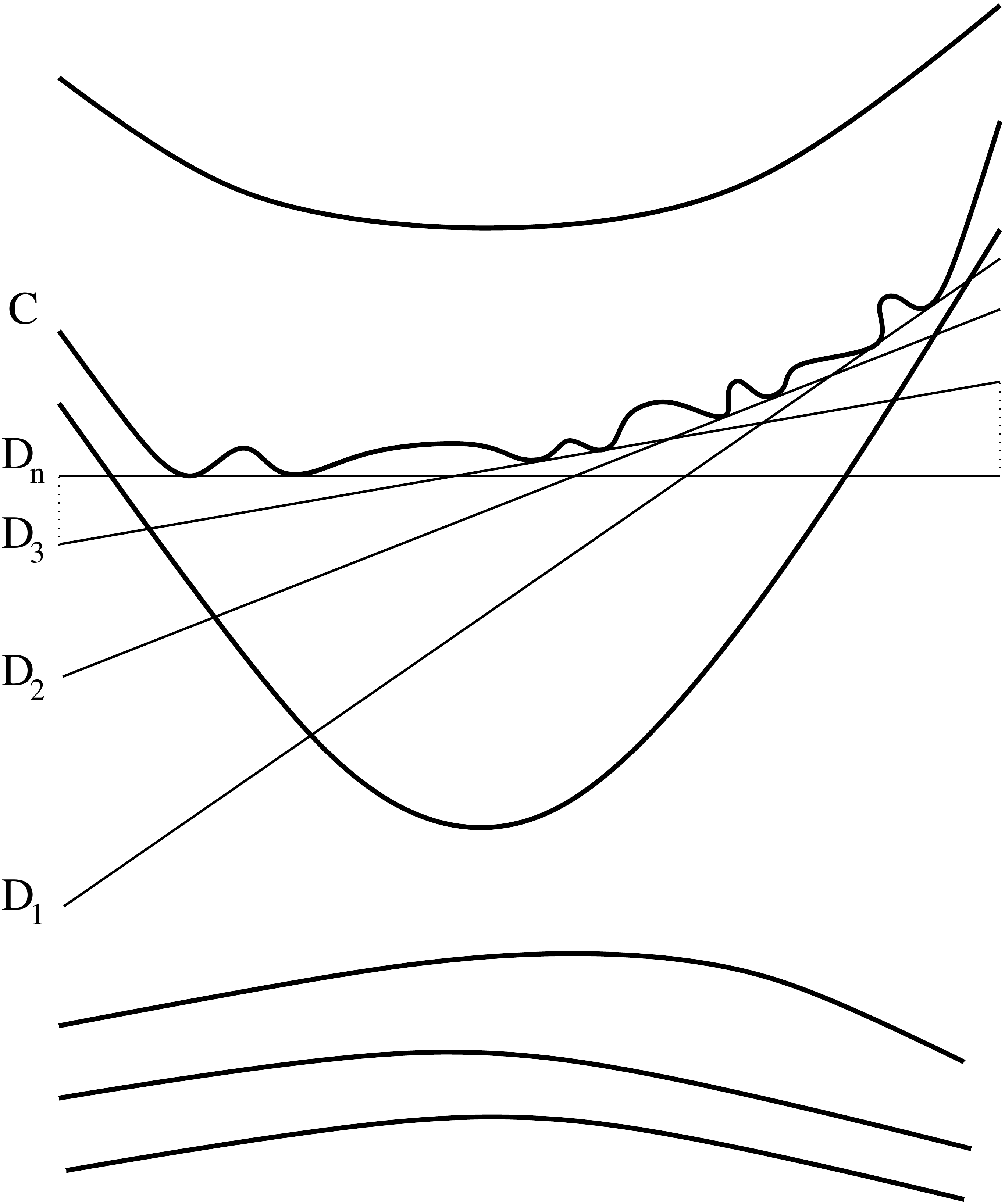}
\\ a) && b)
\end{tabular}
   \caption{}
   \label{fig:lp scheme}
\end{figure}
\begin{lemma}\label{lem:prohib}
Let $C$ be a non-singular maximally inflected hyperbolic curve of
degree 6 in
$\RP$ with exactly $n$ real bitangents 
tangent to $C$  at two
 points on $O_2$ (see Figure \ref{fig:lp scheme}a). Then $n\le 9$.
\end{lemma}
\begin{proof}
The proof uses
the symplectic techniques developed by Orevkov to study the topology of
real plane curves. For the sake of shortness, we do not recall this
technique here. We refer instead
to {\cite[Section 2.1]{O2}} for the definition of a 
real pseudoholomorphic curve in
$\CP$, and to {\cite[Section 3]{O1}} for the definition of a $\mathcal L_p$
scheme and for an account of the braid theoretical methods in the study
of  real plane curves. Note nevertheless that any
real algebraic curve  is a  real
pseudoholomorphic curve.

Choose a point $p$ inside the oval $O_3$.
Applying {\cite[Proposition 3.6]{O1}} if necessary,
 this implies that the $\mathcal L_p$-scheme depicted in Figure
 \ref{fig:lp scheme}b is realizable by a reducible real pseudoholomorphic
 curve of degree $6+n$, whose irreducible components are a sextic
 $C'$
and
 $l$  real pseudoholomorphic
lines $D_1,\ldots, D_n$ intersecting $C'$ transversely in two points
and tangent to $C'$ in two points.
We emphasize that we do not claim that the union of $C$ with its real bitangent
to $O_2$ realizes the $\mathcal L_p$-scheme depicted in Figure
 \ref{fig:lp scheme}b. This is precisely the point where we need to
 consider  real pseudholomorphic curves instead of real algebraic curves, since 
 {\cite[Proposition 3.6]{O1}} does not hold for those latter.

The braid associated to this $\mathcal L_p$-scheme is 
$$b_n=\prod_{i=4}^{n+3} \left[\left( \prod_{j=i}^{n+3}
  \sigma_{j}^{-1}\right) \sigma_{n+4}^{-4} \right]\cdot 
\prod_{j=4}^{n+3}  \sigma_{j}^{-1}\cdot 
\prod_{i=1}^{n+5} \left( \prod_{j=i}^{n+5}
  \sigma_{j}\right) . $$
Computing the Alexander polynomial $p_{10}$ of $b_{10}$, we see that
$\pm i$ is a
simple root of $p_{10}$. 
On the other hand the sum of exponents of
$b_{10}$ equals 15, which is the number of strings of $b_{10}$ minus
1. Hence
  the Murasugi-Tristram inequality {\cite[Section 2.4]{O1}} together
with {\cite[Lemma 2.1]{O1}} implies that the $\mathcal L_p$-scheme depicted in
Figure \ref{fig:lp scheme}b with $n\ge 10$ is not realizable by a real
pseudoholomorphic curve of degree $6+n$.
\end{proof}

\begin{proof}[Proof of Proposition \ref{prop:min 6}]
Let $2n$ be the number of real inflection points of $C$ contained in
$O_2$. According to Lemma \ref{lem:bitangent} the curve $C$ 
has $n$ real bitangents which are
tangent to $C$   at two
 points on $O_2$. Hence we get $n\le 9$ by Lemma \ref{lem:prohib}.
\end{proof}

\subsection{Constructions}
We end the proof of Theorem \ref{theo:hyp deg 6} by showing that
 all distributions of the 24 inflection points on the ovals $O_1$
and $O_2$ which are not
forbidden by Proposition \ref{prop:min 6} are realizable. 
\begin{prop}
For any integer $0\le k\le 9$,
there exists a
 non-singular maximally inflected hyperbolic curve $C(k)$ of
degree 6 in 
$\RP$ such that $O_1$ contains exactly $6+2k$ real inflection points.
\end{prop}
\begin{proof}
The curve $C(9)$ has been constructed in Section
\ref{sec:hilbert2}. The curve $C(k)$ with $k=1,3,5$ and $7$ 
is constructed similarly. Consider two
real conics $C_0$ and $C_1$ intersecting into 4 real points. Choose
$\frac{k-1}{2}$ lines intersecting $C_0$ in $k-1$ points in
$\ext(C_0\cup C_1)$, and $\frac{7-k}{2}$ lines intersecting $C_0$ in
$7-k$ points in 
$\int(C_0\cup C_1)$. Denote by $L$ the union of these lines. 
We obtain a
non-singular maximally
inflected quartic $C_2$ by perturbing   $C_0\cup C_1$  with $L$. By 
 perturbing
$C_0\cup C_2$ to a non-singular maximally
inflected curve of degree 6, we obtain the curve $C(k)$.
\begin{figure}[htbp] 
\begin{tabular}{ccc}
   \centering
   \includegraphics[width=6cm]{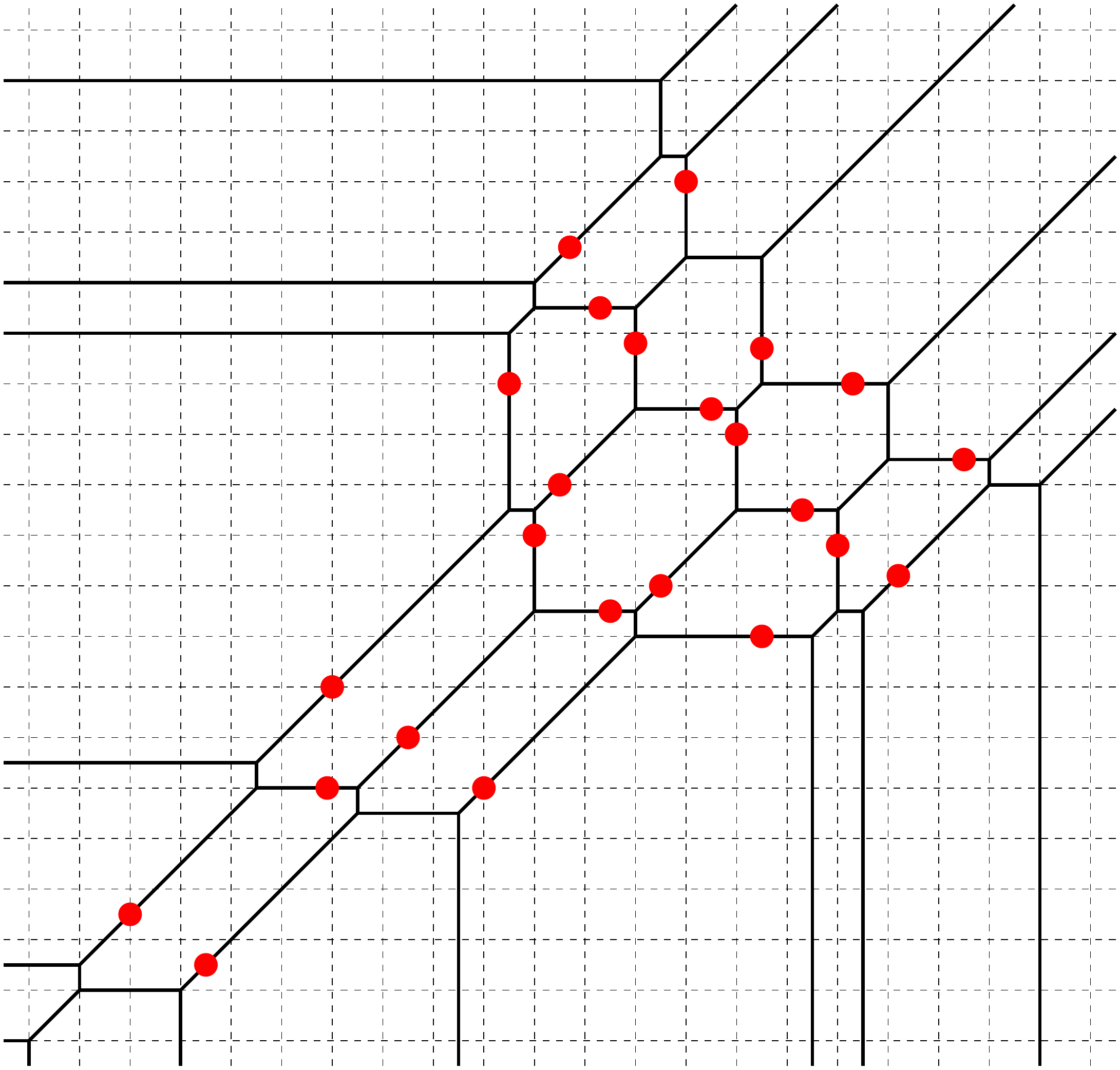} & \hspace{1cm} &
   \includegraphics[width=6cm]{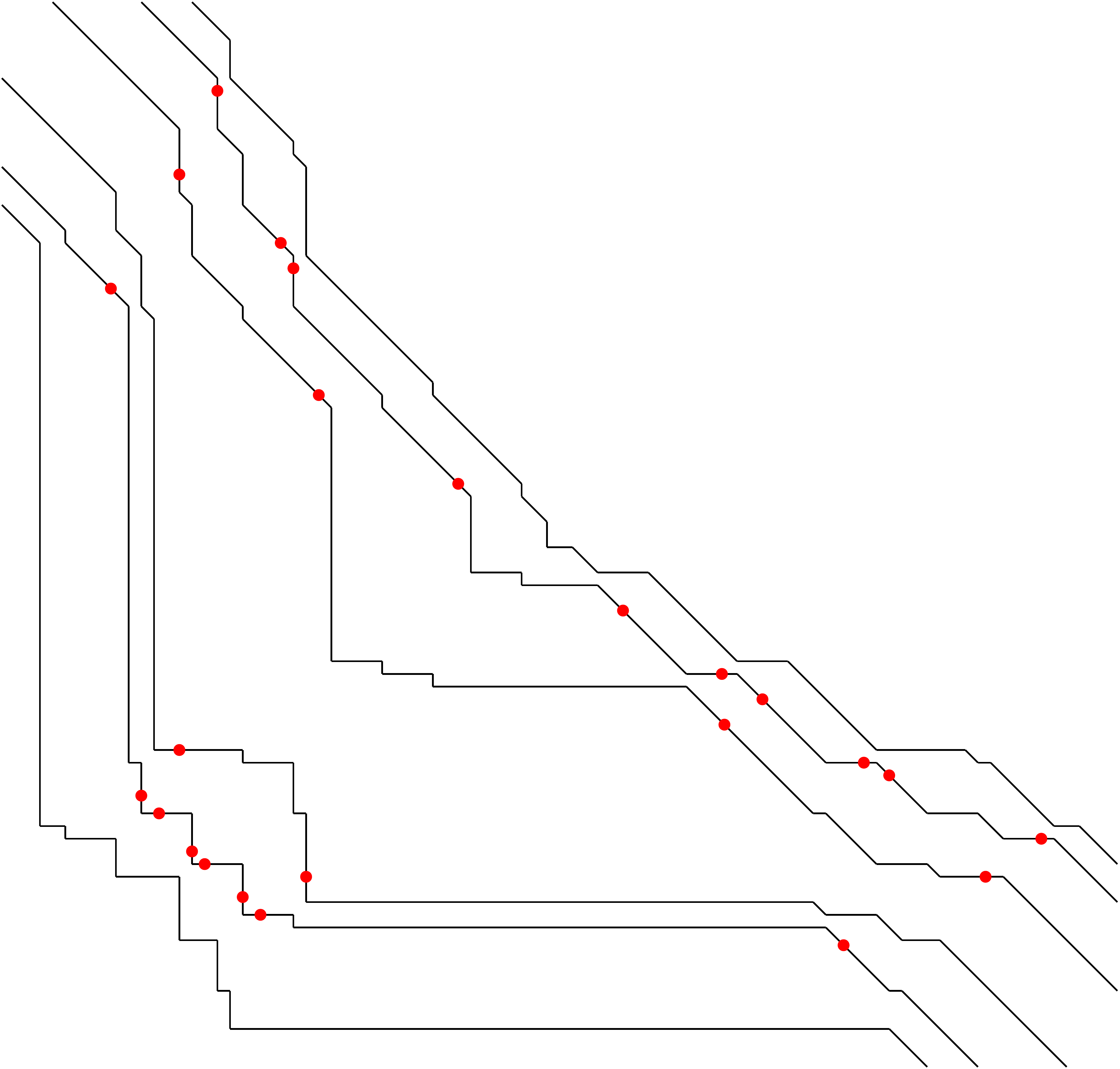}
\\ \\ a) $\mathbb T C$ && b)
\end{tabular}
   \caption{Patchworking of the curve $C(0)$}
   \label{fig:patch}
\end{figure}
\begin{figure}[htbp] 
\begin{tabular}{ccc}
   \centering
   \includegraphics[width=9cm]{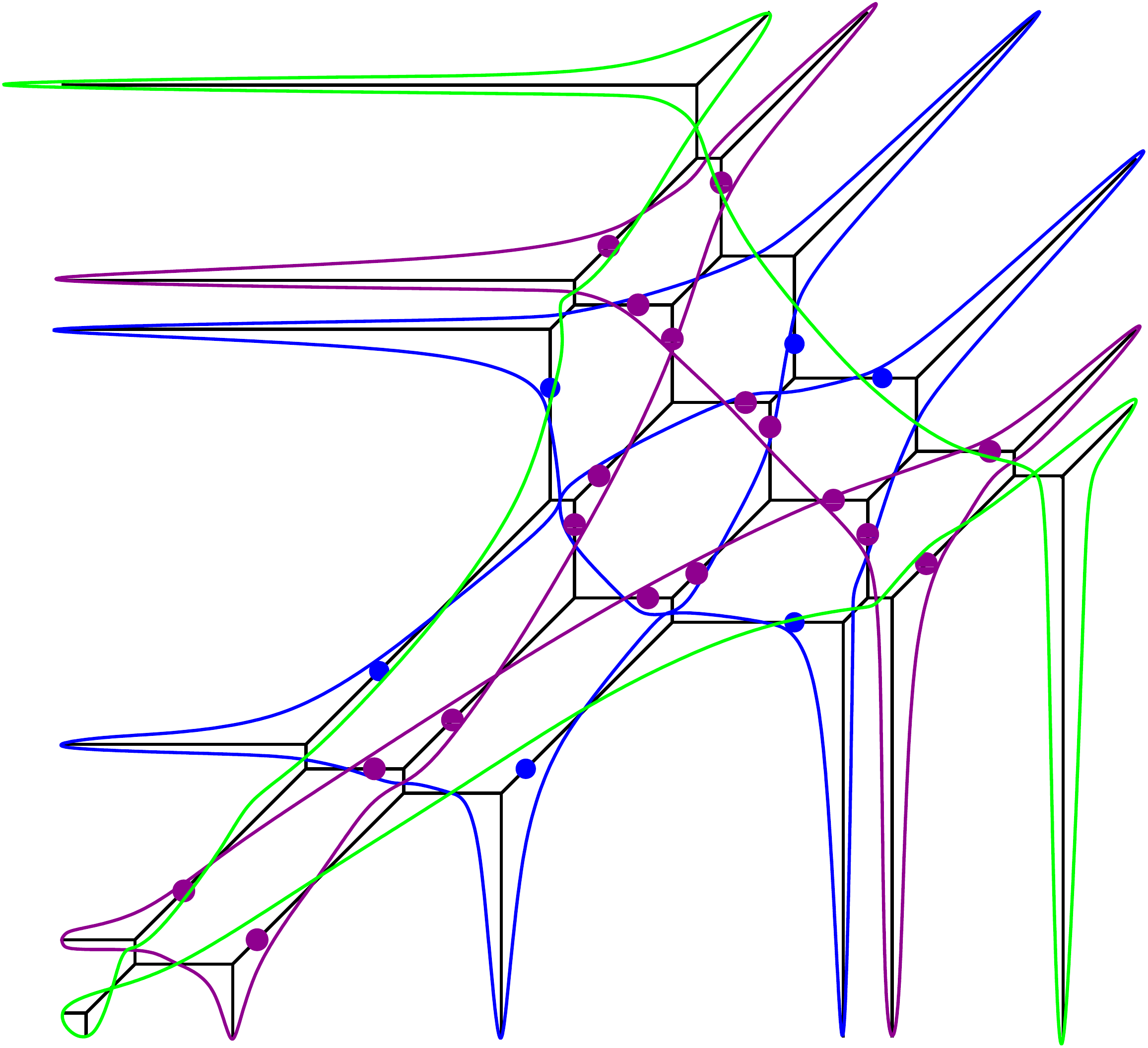}& \hspace{1cm} &
   \includegraphics[width=3cm]{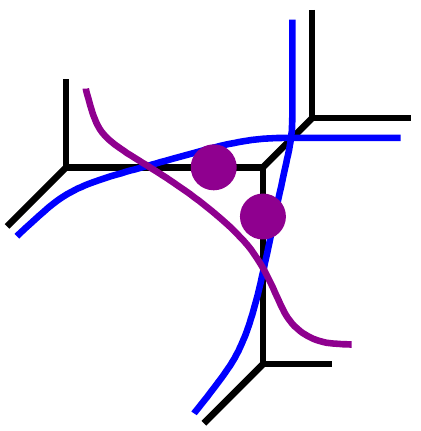}
\\ \\ a)  && b)

\end{tabular}
   \caption{Patchworking of the curve $C(0)$, intermediate step
 (each tropical inflection point  is coloured
according to the arc it comes from)
}
   \label{fig:patch3}
\end{figure}

We construct the curves $C(k)$
with $k$ even by  patchworking. We refer to \cite{BLdM12} for the
definition of tropical
inflection points of tropical plane curves, and  for
the construction of  real algebraic curves with a controlled position of their
real inflection points by patchworking. The tropical curve $\mathbb
TC$ depicted in Figure
\ref{fig:patch}a together with the patchworking depicted in
Figure \ref{fig:patch}b produce the curve $C(0)$. The 
 red 
dots
represent inflection points in 
Figure \ref{fig:patch} (note that all tropical inflection points have
complex multiplicity 3, and real multiplicity 1 in this case). 
Since their position on the tropical curve 
heavily 
depends on the length of its edges, we depicted 
 $\ZZ^2\subset\RR^2$ by the
intersection points of the
 doted lines in Figure \ref{fig:patch}a. 

Let us recall briefly how to locate  real
inflection points in Figure \ref{fig:patch}b. To this aim  we use the ribbon interpretation
of Patchworking, for which one can refer for example to
{\cite[Section 3]{BruBonn}}. The Patchworking depicted in Figure
\ref{fig:patch}b is obtained by twisting all bounded edges of $\mathbb TC$.
We depict in Figure \ref{fig:patch3}a the intermediate picture between 
Figure \ref{fig:patch}a and Figure \ref{fig:patch}b. The  ovals $O_1$, $O_2$,
and $O_3$ of $C(0)$
correspond respectively to the blue, purple, and green curves.
Note that Figure \ref{fig:patch3}a represents the logarithmic image of 
 $C(0)$, in particular the position of $C(0)$ with respect to its
tangent lines cannot be read directly on Figure \ref{fig:patch3}a.
A local computation gives this position for an arc whose logarithmic image is
close to a vertex of $\mathbb TC$ which is not locally a tropical
line. One gets in this way the location of  real inflection points
corresponding to  tropical inflection points located at  vertices of
$\mathbb TC$ (there are no such tropical inflection points in
this particular example, however we are describing the general
procedure to recover the location of real inflection points in a patchworking).
This partial knowledge of the position of $C(0)$ with respect to its
tangent lines determines in its turn the location of real inflections points
corresponding to  tropical inflection points located on edges of
$\mathbb TC$.
For example, any arc whose logarithmic image is
 depicted in Figure  \ref{fig:patch3}b must
contain an even number of real inflection points. Hence these two latter 
must be located on the only arc whose logarithmic image passes close
to the two tropical inflection points.

By varying the length of the edges of $\mathbb TC$, we construct similarly 
 the curves $C(2)$, $C(4)$, and $C(6)$.
Note that any bounded edge of $\mathbb TC$ adjacent to an unbounded
component of $\RR^2\setminus \mathbb TC$ corresponds to a green arc,
which is the logarithmic image of a (convex) arc of $O_3$. In
particular for any choice of the length of edges of $\mathbb
TC$, this construction produces a real hyperbolic curve with at least
six real inflection points  on both $O_1$ and $O_2$.

\begin{figure}[htbp] 
   \centering
\begin{tabular}{c}
   \includegraphics[height=7cm]{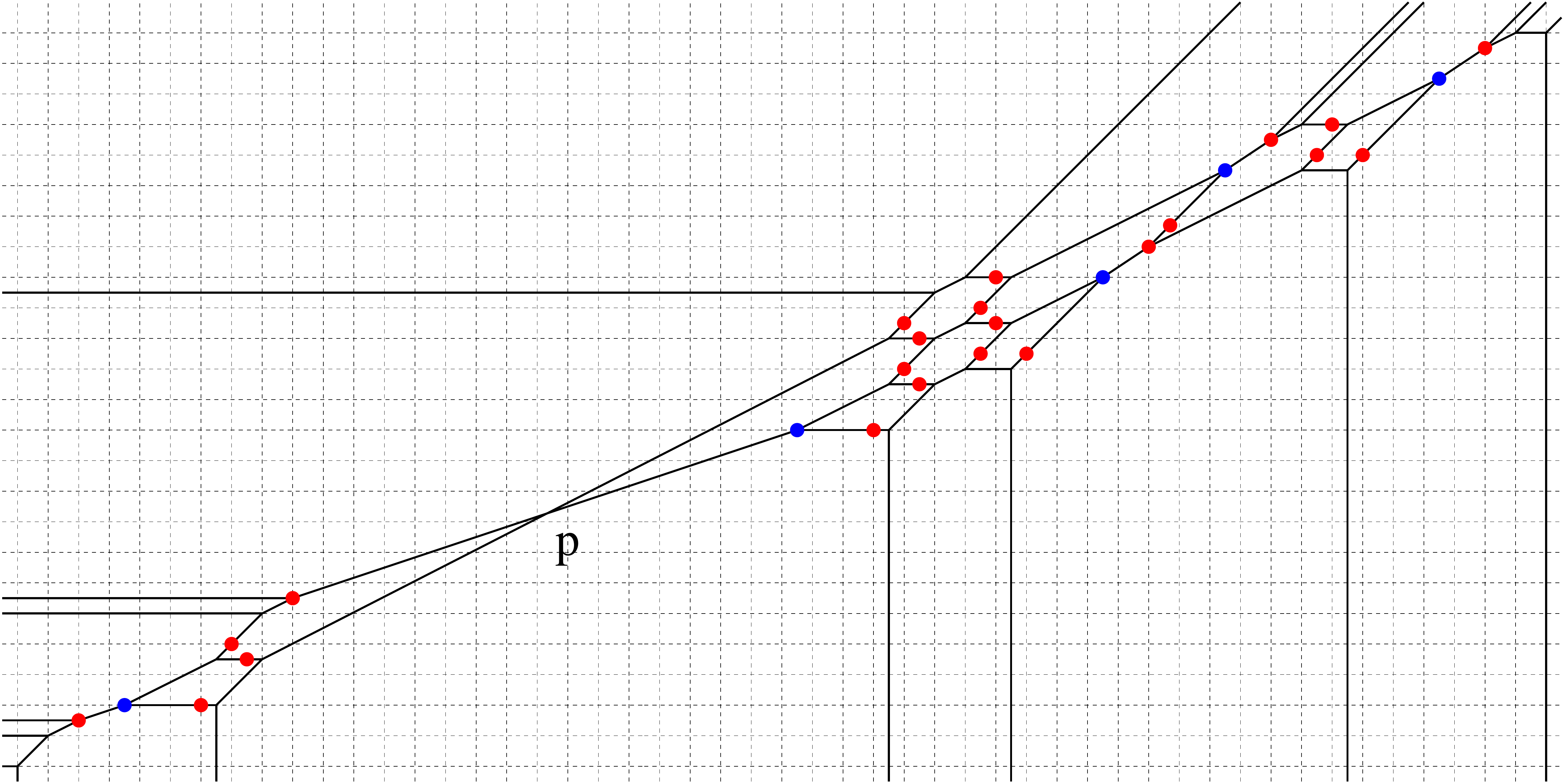}
\\
\\ a)
\\
   \includegraphics[height=7cm]{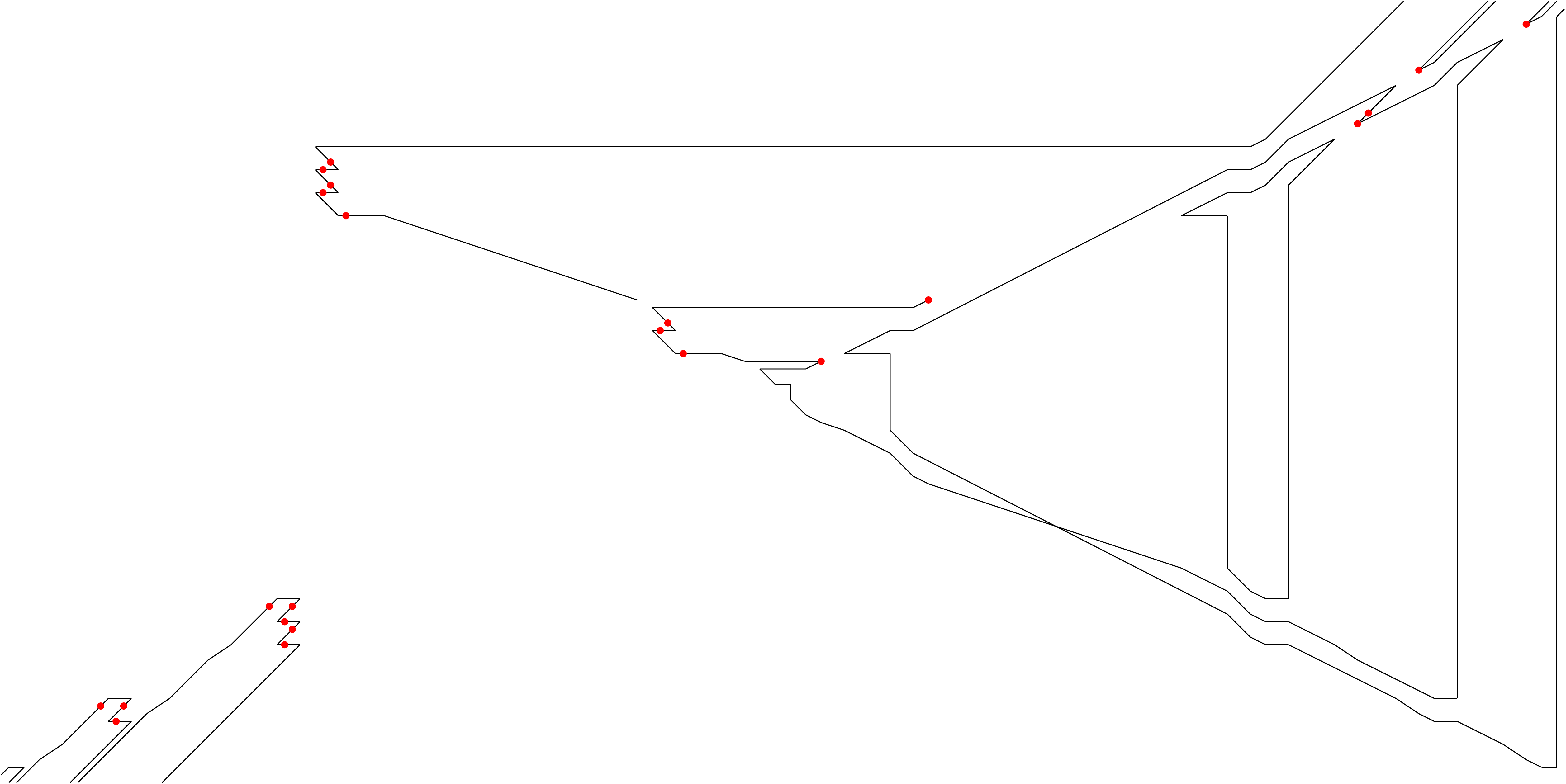}
\\ \\  b)
\end{tabular}
   \caption{Patchworking of the curve $\widetilde C(8)$}
   \label{fig:patch2}
\end{figure}

The construction of the curve $C(8)$ requires a mild adaptation of the
method presented in \cite{BLdM12}. Consider the tropical curve
$\mathbb TC(8)$
 depicted in
Figure \ref{fig:patch2}a. Note that $\mathbb TC(8)$ has a node: all
its vertices are 3-valent, except one 4-valent vertex $p$ which is the
transversal intersection of two edges of $\mathbb TC(8)$. The method used in
\cite{BLdM12} to locate
tropical inflexion points and to compute their real and complex
multiplicities only uses   local data about tropical curves. Since the two
edges of $\mathbb TC(8)$
 intersecting at $p$ do not have any common direction with
an edge of a tropical line, this implies that all computations
performed  in \cite{BLdM12} to locate and compute the multiplicities
of tropical inflexion points of $\mathbb TC(8)$
 distinct from $p$ still holds in
this case. In particular,  red dots in Figure \ref{fig:patch2}a
represent tropical inflexion points distinct from $p$ 
with real multiplicity $1$, and
blue dots   those of real
multiplicity $0$. Using  Shustin's version \cite{Sh1} of Viro's Patchworking
Theorem to construct singular curves,  we can approximate the  real
tropical curve depicted in Figure  \ref{fig:patch2}b by a family
$(\widetilde C_t)_{t\gg1}$ of real
hyperbolic sextic with one generic node. Up to a  translation in $\RR^2$, the
local equation of $\widetilde C_t$ at the node is
$$(x-\alpha y^2)(x+\beta y^3) + o_{t\to +\infty}(t^{-1})$$
with $\alpha$ and $\beta$ two positive real numbers. 
Proposition \ref{prop:klein} implies that 
 smoothing this node 
such that the resulting curve is hyperbolic
produces
 the curve $C(8)$.
\end{proof}

\begin{rem}
If one smoothes ``tropically'' 
the node of the tropical curve $\mathbb TC(8)$ (i.e. perturbing $p$
into two 3-valent vertices before
patchworking),  
 the two additional real
inflection points appear on the oval $O_2$ in the positive
quadrant of $\RR^2$ instead of appearing on the oval $O_1$ in the
quadrant $\{x>0,y<0\}$. This is a
manifestation of the totally discontinuous topology induced on a field
by a non-archimedean valuation. 

On the other hand, one 
can use original Viro's Patchworking to
construct directly the non-singular curve $C(8)$. In this case,
one should use the chart (in the sense of \cite{V1}) at the node $p$ of
$\mathbb TC(8)$ given by a suitable 
perturbation of the curve $(x- y^2)(x+ y^3)$.
\end{rem}

\small

\def\rightmark{\em Bibliography}

\bibliographystyle{alpha}

\newcommand{\etalchar}[1]{$^{#1}$}

\end{document}